\documentclass[a4paper,10pt,reqno]{amsart}
\usepackage{amsmath}
\usepackage{amssymb}
\usepackage{amsthm}
\usepackage[all]{xy}

\theoremstyle{definition}
\newtheorem{Def}{Definition}[section]
\newtheorem{rem}[Def]{Remark}

\newtheorem{ex}[Def]{Example}
\newtheorem*{Conv}{Conventions}

\theoremstyle{plain}
\newtheorem{Prop}[Def]{Proposition}
\newtheorem{Lem}[Def]{Lemma}
\newtheorem{Thm}[Def]{Theorem}
\newtheorem{Cor}[Def]{Corollary}
\newtheorem*{MThm1}{Main Theorem 1}
\newtheorem*{MThm2}{Main Theorem 2}

\def\sub#1{\langle{#1}\rangle_{\rm sub}}
\def\quot#1{\langle{#1}\rangle_{\rm quot}}
\def\ext#1{\langle{#1}\rangle_{\rm ext}}
\def\Serre#1{\langle{#1}\rangle_{\rm Serre}}
\def\loc#1{\langle{#1}\rangle_{\rm loc}}

\def\mbZ{\mathbb{Z}}

\def\mcA{\mathcal{A}}

\def\mcO{\mathcal{O}}

\def\mcS{\mathcal{S}}
\def\mcX{\mathcal{X}}
\def\mcY{\mathcal{Y}}
\def\mcZ{\mathcal{Z}}

\def\mfm{\mathfrak{m}}
\def\mfp{\mathfrak{p}}
\def\mfq{\mathfrak{q}}

\makeatletter
\def\Mod{\mathop{\operator@font Mod}\nolimits}
\def\mod{\mathop{\operator@font mod}\nolimits}
\def\noeth{\mathop{\operator@font noeth}\nolimits}
\def\Hom{\mathop{\operator@font Hom}\nolimits}
\def\Ker{\mathop{\operator@font Ker}\nolimits}
\def\Im{\mathop{\operator@font Im}\nolimits}
\def\Cok{\mathop{\operator@font Cok}\nolimits}

\def\Spec{\mathop{\operator@font Spec}\nolimits}
\def\Supp{\mathop{\operator@font Supp}\nolimits}
\def\Ass{\mathop{\operator@font Ass}\nolimits}
\def\ASpec{\mathop{\operator@font ASpec}\nolimits}
\def\AISpec{\mathop{\operator@font ASpec_{0}}\nolimits}
\def\ASupp{\mathop{\operator@font ASupp}\nolimits}
\def\AAss{\mathop{\operator@font AAss}\nolimits}

\def\Zg{\mathop{\operator@font Zg}\nolimits}
\def\Ann{\mathop{\operator@font Ann}\nolimits}
\makeatother

\title[Classifying Serre subcategories via atom spectrum]{Classifying Serre subcategories via atom spectrum}
\author{Ryo Kanda}
\address{Graduate School of Mathematics, Nagoya University, Furo-cho, Chikusa-ku, Nagoya-shi, Aichi-ken, 464-8602, Japan}
\email{kanda.ryo@a.mbox.nagoya-u.ac.jp}

\begin{document}

\begin{abstract}
	In this paper, we introduce the atom spectrum of an abelian category as a topological space consisting of all the equivalence classes of monoform objects. In terms of the atom spectrum, we give a classification of Serre subcategories of an arbitrary noetherian abelian category. Moreover we show that the atom spectrum of a locally noetherian Grothendieck category is homeomorphic to its Ziegler spectrum.
\end{abstract}

\maketitle

\section{Introduction}
\label{sec:intro}

Classification of localizing subcategories and Serre subcategories is quite an active subject widely studied by a number of authors (see, for example, \cite{Gabriel}, \cite{Herzog}, \cite{Krause}, \cite{Takahashi}, and \cite{BensonIyengarKrause}). The prototype of such classifications is a result given by Gabriel \cite{Gabriel}: for a right noetherian ring $R$, denote by $\Mod R$ the category of right $R$-modules and by $\mod R$ the category of finitely generated right $R$-modules.

\begin{Thm}[Gabriel \cite{Gabriel}]
	Let $R$ be a commutative noetherian ring. Then there exist one-to-one correspondences between localizing subcategories of $\Mod R$, Serre subcategories of $\mod R$, and specialization-closed subsets of $\Spec R$.
\end{Thm}

In the paper \cite{Gabriel}, Gabriel used the space of isomorphism classes of indecomposable injective modules in order to show the above theorem. On the other hand, Herzog \cite{Herzog} constructed one-to-one correspondences between localizing subcategories of $\Mod R$, Serre subcategories of $\mod R$, and open subsets of the Ziegler spectrum of $\Mod R$. The Ziegler spectrum was originally introduced by Ziegler \cite{Ziegler}. The Ziegler spectrum of a locally coherent Grothendieck category is a topological space consisting of all the isomorphism classes of indecomposable injective objects (Definition \ref{Def:Zieglerspectrum}). The correspondence between the Ziegler spectrum and Serre subcategories is also discussed in \cite{Krause}.

The aim of this paper is to give a classification of Serre subcategories of an arbitrary noetherian abelian category in terms of the atom spectrum (Definition \ref{Def:atomspectrum}), which is a topological space consisting of all the equivalence classes of monoform objects (Definition \ref{Def:monoform}). Monoform objects and their equivalence relation are investigated in the context of noncommutative ring theory (see, for example, \cite{Storrer}, \cite{LambekMichler}, and \cite{Dlab}). Note that these notions are called by various names. For example, in \cite{Storrer}, monoform objects are called strongly uniform objects. For more details, we refer the reader to \cite{Reyes}.

Now we state our first main result in this paper. We denote by $\ASupp M$ the atom support (Definition \ref{Def:atomsupport}) of an object $M$ in an abelian category $\mcA$. 

\begin{MThm1}[Theorem \ref{Thm:main}]
	Let $\mcA$ be a noetherian abelian category. Then the map $\mcX\mapsto\bigcup_{M\in\mcX}\ASupp M$ gives a one-to-one correspondence between Serre subcategories of $\mcA$ and open subsets of the atom spectrum of $\mcA$. The inverse is given by $\Phi\mapsto\{M\in\mcA\ |\ \ASupp M\subset\Phi\}$.
\end{MThm1}

One of the advantages of dealing with the atom spectrum is that in Main Theorem 1, we do not have to assume that $\mcA$ has enough injectives in contrary to Herzog's correspondence. In fact, we do not use injective objects in order to prove Main Theorem 1.

Now let $\mcA$ be a locally noetherian Grothendieck category. It is shown in \cite{Storrer} that, as a set, the atom spectrum coincides with the Ziegler spectrum. Our second main result shows that they are the same as topological spaces. Moreover we have the following classification of subcategories: denote by $\noeth\mcA$ the subcategory of noetherian objects in $\mcA$.

\begin{MThm2}[Theorem \ref{Thm:atomspecZg} and Theorem \ref{Thm:locSerreatom}]
	Let $\mcA$ be a locally noetherian Grothendieck category.
	\begin{enumerate}
		\item The atom spectrum of $\mcA$ is homeomorphic to the Ziegler spectrum of $\mcA$.
		\item There exist one-to-one correspondences between localizing subcategories of $\mcA$, Serre subcategories of $\noeth\mcA$, and open subsets of the atom spectrum of $\mcA$.
	\end{enumerate}
\end{MThm2}

In the case where $\mcA=\mod R$ for a noetherian ring $R$, each equivalence class of monoform objects is represented by $R/\mfp$, where $\mfp$ is a comonoform right ideal (Definition \ref{Def:comonoformrightideal}) of $R$. If $R$ is commutative, comonoform right ideals of $R$ are nothing but prime ideals of $R$. Thus Main Theorem 2 (1) can be interpreted as a noncommutative generalization of the correspondence given by Matlis \cite{Matlis}. Moreover Gabriel's theorem is recovered from our Main Theorem 2 (2). These facts suggest that the atom spectrum is a reasonable noncommutative generalization of the prime spectrum of a commutative noetherian ring.

\begin{Conv}
	Throughout this paper, a {\em subcategory} means a full subcategory which contains zero objects and is closed under isomorphisms. Unless otherwise specified, $\mcA$ is an abelian category, and an {\em object} means an object in $\mcA$. We denote by $R$ an associative right noetherian ring with an identity element. A {\em module} means a right $R$-module. We denote the category of right $R$-modules by $\Mod R$ and the category of finitely generated right $R$-modules by $\mod R$.
\end{Conv}

\section{Monoform objects and their relationships to Serre subcategories}
\label{sec:Serre}

Throughout this section, let $\mcA$ be an abelian category. In this section, we recall some properties of monoform objects, which are mainly stated in \cite{Storrer}. First of all, we define monoform objects in $\mcA$, which play central roles throughout this paper.

\begin{Def}\label{Def:monoform}
	Let $\mcA$ be an abelian category. Call a nonzero object $H$ in $\mcA$ a {\em monoform object} if for any nonzero subobject $N$ of $H$, there exists no common nonzero subobject of $H$ and $H/N$, that is, there does not exist a nonzero subobject of $H$ which is isomorphic to a subobject of $H/N$.
\end{Def}

Note that any simple object is a monoform object.

The following fact about monoform objects is frequently used later.

\begin{Prop}\label{Prop:monoformsub}
	Any nonzero subobject of a monoform object is also a monoform object.
\end{Prop}

\begin{proof}
	Let $H$ be a monoform object and $L$ be a nonzero subobject of $H$. If $L$ is not monoform, there exists a nonzero subobject $N$ of $L$ such that $L$ and $L/N$ have a common nonzero subobject. It is a common nonzero subobject of $H$ and $H/N$. This is a contradiction. Therefore $L$ is a monoform object.
\end{proof}

We use the following notation about subcategories of $\mcA$.

\begin{Def}
	Let $\mcX$ and $\mcY$ be subcategories of $\mcA$.
	\begin{enumerate}
		\item Set the subcategory $\sub{\mcX}$ of $\mcA$ by
		\begin{eqnarray*}
			\sub{\mcX}=\{M\in\mcA\ |\ M \text{ is a subobject of an object in } \mcX\}.
		\end{eqnarray*}
		In the case where $\sub{\mcX}=\mcX$, we say that $\mcX$ is {\em closed under subobjects}.
		\item Set the subcategory $\quot{\mcX}$ of $\mcA$ by
		\begin{eqnarray*}
			\quot{\mcX}=\{M\in\mcA\ |\ M \text{ is a quotient object of an object in } \mcX\}.
		\end{eqnarray*}
		In the case where $\quot{\mcX}=\mcX$, we say that $\mcX$ is {\em closed under quotient objects}.
		\item Set the subcategory $\mcX * \mcY$ of $\mcA$ by
		\begin{eqnarray*}
			\mcX * \mcY=\{M\in\mcA\!\!\!&|&\!\!\!\text{there exists an exact sequence }\\
			&&\!\!\!0\to L\to M\to N\to 0\\
			&&\!\!\!\text{with }L\in \mcX\text{ and }N\in \mcY\}.
		\end{eqnarray*}
		In the case where $\mcX * \mcX=\mcX$, we say that $\mcX$ is {\em closed under extensions}.
		\item Set $\mcX^{0}=\{\text{zero objects}\}$ and $\mcX^{n+1}=\mcX^{n}*\mcX$ for $n\in \mbZ_{\geq 0}$.
		Set the subcategory $\ext{\mcX}$ of $\mcA$ by
		\begin{eqnarray*}
			\ext{\mcX}=\bigcup_{n\geq 0}\mcX^{n}.
		\end{eqnarray*}
		\item $\mcX$ is called a {\em Serre subcategory} of $\mcA$ if $\mcX$ is closed under subobjects, quotient objects, and extensions. Let $\Serre{\mcX}$ denote the smallest Serre subcategory which contains $\mcX$.
		\item $\mcX$ is called a {\em localizing subcategory} of $\mcA$ if $\mcX$ is a Serre subcategory which is {\em closed under arbitrary direct sums}, that is, for any family $\{M_{\lambda}\}_{\lambda\in\Lambda}$ of objects in $\mcX$, its direct sum $\bigoplus_{\lambda\in\Lambda}M_{\lambda}$ also belongs to $\mcX$ if it exists in $\mcA$. Let $\loc{\mcX}$ denote the smallest localizing subcategory which contains $\mcX$.
	\end{enumerate}
\end{Def}

In order to state the relationship between monoform objects and Serre subcategories, we recall some properties about subcategories.

\begin{Prop}\label{Prop:subcats}
	Let $\mcX$, $\mcY$, and $\mcZ$ be subcategories of $\mcA$.
	\begin{enumerate}
		\item $\quot{\sub{\mcX}}=\sub{\quot{\mcX}}$.
		\item $(\mcX*\mcY)*\mcZ=\mcX*(\mcY*\mcZ)$.
		\item $\ext{\mcX}$ is the smallest extension-closed subcategory which contains $\mcX$.
		\item $\sub{\mcX*\mcY}\subset\sub{\mcX}*\sub{\mcY}$,
		$\quot{\mcX*\mcY}\subset\quot{\mcX}*\quot{\mcY}$.
		\item $\sub{\ext{\mcX}}\subset\ext{\sub{\mcX}}$,
		$\quot{\ext{\mcX}}\subset\ext{\quot{\mcX}}$.
		\item $\ext{\quot{\sub{\mcX}}}=\Serre{\mcX}$.
	\end{enumerate}
\end{Prop}

\begin{proof}
	(1) A quotient object of a subobject of an object $M$ is of the form $L/N$, where $L$ and $N$ are subobject of $M$ such that $N\subset L$. So is a subobject of a quotient object of $M$.
	
	(2) Assume that $W$ belongs to $(\mcX*\mcY)*\mcZ$. Then there exist exact sequences $0\to X\to V\to Y\to 0$ and $0\to V\to W\to Z\to 0$ such that $X\in \mcX$, $Y\in \mcY$, and $Z\in \mcZ$. Denote the cokernel of the composite $X\to V\to W$ by $U$. Then we obtain the following commutative diagram by snake lemma:
	\begin{eqnarray*}
		\xymatrix{
			& 0\ar[d] & 0\ar[d] & 0\ar[d] & \\
			0\ar[r] & X\ar[r]\ar@{=}[d] & V\ar[r]\ar[d] & Y\ar[r]\ar[d] & 0 \\
			0\ar[r] & X\ar[r]\ar[d] & W\ar[r]\ar[d] & U\ar[r]\ar[d]& 0 \\
			 & 0\ar[r] & Z\ar@{=}[r]\ar[d] & Z\ar[r]\ar[d] & 0 \\
			&& 0 & 0 &.
		}
	\end{eqnarray*}
	This means that $W$ belongs to $\mcX*(\mcY*\mcZ)$, and we have $(\mcX*\mcY)*\mcZ\subset\mcX*(\mcY*\mcZ)$. In a similar way, we obtain $(\mcX*\mcY)*\mcZ\supset\mcX*(\mcY*\mcZ)$.
	
	(3) This follows from the definition of $\ext{\cdot}$ and (2).
	
	(4) Let $0\to L \to M \to N \to 0$ be an exact sequence and $M\in \mcX*\mcY$. Then there exist objects $M'$ in $\mcX$, $M''$ in $\mcY$, and an exact sequence $0\to M' \to M \to M''\to 0$. Denote the image of the composite $M'\to M \to N$ by $N'$ and the kernel of $M'\to N'$ by $L'$. Then we obtain the following commutative diagram by the snake lemma:
	\begin{eqnarray*}
		\xymatrix{
			& 0\ar[d] & 0\ar[d] & 0\ar[d] & \\
			0\ar[r] & L'\ar[d]\ar[r] & M'\ar[d]\ar[r] & N'\ar[d]\ar[r] & 0 \\
			0\ar[r] & L \ar[d]\ar[r]& M\ar[d]\ar[r] & N\ar[d]\ar[r] & 0 \\
			0\ar[r] & L''\ar[d]\ar[r] & M''\ar[d]\ar[r] & N''\ar[d]\ar[r] & 0 \\
			& 0 & 0 & 0 & .
		}
	\end{eqnarray*}
	This diagram shows that $L$ belongs to $\sub{\mcX}*\sub{\mcY}$ and that $N$ belongs to $\quot{\mcX}*\quot{\mcY}$.
	
	(5) This follows from the definition of $\ext{\cdot}$ and (4).
	
	(6) It is enough to show that $\ext{\quot{\sub{\mcX}}}$ is a Serre subcategory. By (3), $\ext{\quot{\sub{\mcX}}}$ is closed under extensions. By (1) and (5), we have
	\begin{eqnarray*}
		\sub{\ext{\quot{\sub{\mcX}}}}\subset\ext{\sub{\quot{\sub{\mcX}}}}=\ext{\quot{\sub{\mcX}}},\\
		\quot{\ext{\quot{\sub{\mcX}}}}\subset\ext{\quot{\quot{\sub{\mcX}}}}=\ext{\quot{\sub{\mcX}}},
	\end{eqnarray*}
	and hence $\ext{\quot{\sub{\mcX}}}$ is also closed under subobjects and quotient objects.
\end{proof}

According to Proposition \ref{Prop:subcats} (1), a subobject of a quotient object of an object $M$ is called a {\em subquotient} of $M$.

We obtain the following characterization of monoform modules.

\begin{Thm}\label{Thm:monoformness}
	Let $\mcA$ be an abelian category. Then an object $M$ in $\mcA$ is monoform if and only if $M$ does not belong to the smallest Serre subcategory which contains all the objects of the form $M/N$, where $N$ is a nonzero subobject of $M$.
\end{Thm}

\begin{proof}
	Set $\mcX=\Serre{M/N\ |\ N\text{ is a nonzero subobject of }M}$.
	
	Let $M$ be a non-monoform object. Then there exist nonzero subobjects $N$, $L$, and $L'$ of $M$ such that $L\subset L'$, and $N\cong L'/L$. Since $L'/L$ and $M/N$ belong to $\mcX$, we deduce that $M$ also belongs to $\mcX$ from the exact sequence
	\begin{eqnarray*}
		0\to \frac{L'}{L}\to M \to \frac{M}{N}\to 0.
	\end{eqnarray*}
	
	Conversely, assume that $M$ belongs to $\mcX$. By Proposition \ref{Prop:subcats} (6), we have
	\begin{eqnarray*}
		\mcX&=&\bigg\langle\bigg\langle\bigg\langle
		\frac{M}{N}\ \bigg|\ N\text{ is a nonzero subobject of }M
		\bigg\rangle_{\rm sub}\bigg\rangle_{\rm quot}\bigg\rangle_{\rm ext}\\
		&=&\ext{\mcY},
	\end{eqnarray*}
	where $\mcY=\{L'/L\ |\ L\text{ and }L'\text{ are nonzero subobjects of }M\text{ such that }L\subset L'$\}.
	Then there exists $n\in\mbZ_{\geq 0}$ such that $M$ belongs to $\mcY^{n}$ and does not belong to $\mcY^{n-1}$.
	In the case where $n=0$, $M$ is a zero object, and hence $M$ is not monoform. In the case where $n\geq 1$, we have an exact sequence
	\begin{eqnarray*}
		0 \to \frac{L'}{L} \to M \to N \to 0,
	\end{eqnarray*}
	where $L$ and $L'$ are nonzero subobjects of $M$ such that $L\subset L'$, and $N$ belongs to $\mcY^{n-1}$. Since $M$ does not belong to $\mcY^{n-1}$, we have $M\not\cong N$, and hence $L'/L$ is a nonzero subobject of $M$. This means that $M$ is not monoform.
\end{proof}

Now we define an equivalence relation between monoform objects. In order to do that, we pay attention to the following fact. Recall that a nonzero object $M$ is called a {\em uniform object} if for any nonzero subobjects $L$ and $L'$ of $M$, we have $L\cap L'\neq 0$.

\begin{Prop}\label{Prop:monoformuniform}
	Any monoform object is a uniform object.
\end{Prop}

\begin{proof}
	Let $H$ be a monoform object, and assume that there exist nonzero subobjects $L$ and $L'$ of $H$ such that $L\cap L'=0$. Then the sum $L+L'$ is isomorphic to the direct sum $L\oplus L'$, and by Proposition \ref{Prop:monoformsub}, it is also a monoform object. However, $(L\oplus L')/L'$ is isomorphic to $L$. This is a contradiction. Therefore $H$ is a uniform object.
\end{proof}

\begin{Def}\label{Def:atomequivalence}
	We say that monoform objects $H$ and $H'$ are {\em atom-equivalent} if there exists a common nonzero subobject of $H$ and $H'$.
\end{Def}

\begin{Prop}\label{Prop:atomequiv}
	The atom equivalence is an equivalence relation between monoform objects.
\end{Prop}

\begin{proof}
	Only the transitivity is non-trivial. Let $H$, $H'$, and $H''$ be monoform objects, and assume that $H$ is atom-equivalent to $H'$ and that $H'$ is atom-equivalent to $H''$. Then there exist nonzero subobjects $L_{1}$ and $L_{2}$ of $H'$ such that $L_{1}$ is also a subobject of $H$, and $L_{2}$ is also that of $H''$. By Proposition \ref{Prop:monoformuniform}, $H'$ is a uniform object, and hence $L_{1}$ and $L_{2}$ have a nonzero intersection $L$ in $H'$. Then $L$ is a common nonzero subobject of $H$ and $H''$.
\end{proof}

Now we show that any nonzero noetherian object has a monoform subobject. Recall that an object $M$ is called {\em noetherian} if for any ascending chain
\begin{eqnarray*}
	L_{0}\subset L_{1}\subset L_{2}\subset \cdots
\end{eqnarray*}
of subobjects of $M$, there exists $n\in\mbZ_{\geq 0}$ such that
\begin{eqnarray*}
	L_{n}= L_{n+1}= L_{n+2}= \cdots.
\end{eqnarray*}

\begin{Thm}\label{Thm:monoformfiltration}
	Let $\mcA$ be an abelian category. Then for any noetherian object $M$ in $\mcA$, there exists a filtration
	\begin{eqnarray*}
		0=L_{0}\subset L_{1}\subset\cdots\subset L_{n}=M
	\end{eqnarray*}
	such that $L_{i}/L_{i-1}$ is a monoform object for any $i=1,\ldots,n$. In particular, $M$ has a monoform subobject $L_{1}$.
\end{Thm}

\begin{proof}
	Set $\mcX=\{\text{zero objects}\}\cup\{\text{monoform objects}\}$. Remark that $\sub{\mcX}=\mcX$ by Proposition \ref{Prop:monoformsub}. By the definition of $\ext{\cdot}$, it is enough to show that $M$ belongs to $\ext{\mcX}$. Assume that $M$ does not belong to $\ext{\mcX}$. Since $M$ is noetherian, there exists a maximal element $N$ in
	\begin{eqnarray*}
		\{M'\in\mcA\ |\ M'\text{ is a subobject of }M\text{ such that }\frac{M}{M'}\\\text{does not belong to }\ext{\mcX}\}.
	\end{eqnarray*}
	Since $M/N$ does not belong to $\ext{\mcX}$, $M/N$ is not monoform. Hence there exist subobjects $N'$, $L$, and $L'$ of $M$ such that $N\subsetneq N'$, $N\subsetneq L\subsetneq L'$, and
	\begin{eqnarray*}
		\frac{N'}{N}\cong\frac{L'}{L}.
	\end{eqnarray*}
	By the maximality of $N$, $M/N'$ and $M/L$ belong to $\ext{\mcX}$. By Proposition \ref{Prop:subcats} (5), we have $L'/L\in\sub{\ext{\mcX}}\subset\ext{\sub{\mcX}}=\ext{\mcX}$. The exact sequence
	\begin{eqnarray*}
		0\to \frac{L'}{L}\to\frac{M}{N}\to\frac{M}{N'}\to 0
	\end{eqnarray*}
	implies that $M/N$ also belongs to $\ext{\mcX}$. This is a contradiction.
\end{proof}

In the rest of this section, by investigating some relationships between monoform objects and uniform objects, we show that any noetherian uniform object has a unique maximal monoform subobject. These arguments are not used in the proofs of the main theorems.

\begin{Prop}\label{Prop:quotuniform}
	Let $M$ be a noetherian object and $L$ be a uniform subobject of $M$. Then there exists a subobject $N$ of $M$ such that the composite $L\hookrightarrow M\twoheadrightarrow M/N$ is a monomorphism, and $M/N$ is a uniform object.
\end{Prop}

\begin{proof}
	Take a maximal element $N$ in
	\begin{eqnarray*}
		\{M'\in\mcA\ |\ M'\text{ is a subobject of }M\text{ such that }L\cap M'=0\}.
	\end{eqnarray*}
	Then the kernel of the composite $L\hookrightarrow M\twoheadrightarrow M/N$ is $L\cap N=0$. Assume that $M/N$ is not uniform. Then there exist subobjects $N_{1}$ and $N_{2}$ of $M$ such that $N\subsetneq N_{1}$, $N\subsetneq N_{2}$, and $N_{1}\cap N_{2}=N$. By the maximality of $N$, $L\cap N_{1}$ and $L\cap N_{2}$ are nonzero. Since $L$ is uniform, we have $L\cap N=(L\cap N_{1})\cap (L\cap N_{2})\neq 0$. This is a contradiction. Therefore $M/N$ is uniform.
\end{proof}

\begin{Prop}\label{Prop:monoformsum}
	Let $M$ be a noetherian uniform object and $H_{1},H_{2}$ be monoform subobjects of $M$. Then $H_{1}+H_{2}$ is a monoform subobject of $M$.
\end{Prop}

\begin{proof}
	Assume that $H_{1}+H_{2}$ is not monoform. Then there exists a nonzero subobject $N$ of $H_{1}+H_{2}$ such that $H_{1}+H_{2}$ and $(H_{1}+H_{2})/N$ have a common nonzero subobject $L$. Since $M$ is uniform, $L$ is also uniform. Hence by applying Proposition \ref{Prop:quotuniform} to $L\subset(H_{1}+H_{2})/N$, we obtain a subobject $N'$ of $H_{1}+H_{2}$ such that $N\subset N'$ holds, $L$ is isomorphic to a subobject of $(H_{1}+H_{2})/N'$, and $(H_{1}+H_{2})/N'$ is uniform. Since
	\begin{eqnarray*}
		\frac{H_{1}+H_{2}}{N'}=\frac{H_{1}+N'}{N'}+\frac{H_{2}+N'}{N'},
	\end{eqnarray*}
	there exists $i=1,2$ such that $(H_{i}+N')/N'\neq 0$. Since $(H_{1}+H_{2})/N'$ is uniform, $L$ and $(H_{i}+N')/N'$ have a nonzero intersection in $(H_{1}+H_{2})/N'$. By replacing $L$ by the intersection, we can assume that $L$ is a subobject of $(H_{i}+N')/N'$. Similarly, since $M$ is uniform, we can assume that $L$ is a subobject of $H_{i}$. Since $H_{i}$ is a monoform object, and we have
	\begin{eqnarray*}
		\frac{H_{i}+N'}{N'}\cong\frac{H_{i}}{H_{i}\cap N'},
	\end{eqnarray*}
	we obtain $H_{i}\cap N'=0$. This contradicts the fact that $M$ is uniform. Therefore $H_{1}+H_{2}$ is a monoform subobject of $M$.
\end{proof}

\begin{Thm}\label{Thm:maxmonoform}
	Let $\mcA$ be an abelian category. Then any noetherian uniform object in $\mcA$ has a unique maximal monoform subobject.
\end{Thm}

\begin{proof}
	Let $M$ be a noetherian uniform object. By Theorem \ref{Thm:monoformfiltration}, there exists at least one monoform subobject of $M$, and hence there exists a maximal monoform subobject since $M$ is noetherian. Let $H_{1}$ and $H_{2}$ be maximal monoform subobjects of $M$. Then by Proposition \ref{Prop:monoformsum}, $H_{1}+H_{2}$ is also monoform. Then by the maximality of $H_{1}$ and $H_{2}$, we have $H_{1}=H_{1}+H_{2}=H_{2}$.
\end{proof}

\section{Atom spectrum}
\label{sec:spectrum}

In this section, we introduce the atom spectrum of an abelian category $\mcA$, which is a topological space associated with $\mcA$. Although it is not necessarily a set, we call it a ``space'' for simplicity.

\begin{Def}\label{Def:atomspectrum}
	Let $\mcA$ be an abelian category. Denote the class of all the monoform objects in $\mcA$ by $\AISpec \mcA$. The {\em atom spectrum} $\ASpec \mcA$ of $\mcA$ is the quotient class of $\AISpec \mcA$ by the atom equivalence. The equivalence class of a monoform object $H$ is denoted by $\overline{H}$ and is called the {\em atom} of $H$.
\end{Def}

As we see in section \ref{sec:commutative}, the atom spectrum is in fact a generalization of the prime spectrum of a commutative noetherian ring. In the rest of this section, we introduce the atom support and the associated atoms of an object. They are generalizations of the support and the associated primes of a module over a commutative noetherian ring, respectively.

\begin{Def}\label{Def:atomsupport}
	For an object $M$, we define a subclass $\ASupp M$ of $\ASpec\mcA$ by
	\begin{eqnarray*}
		\ASupp M=\{\overline{H}\in\ASpec\mcA\ |\text{ there exists }H'\in\overline{H}\\\text{which is a subquotient of }M\}.
	\end{eqnarray*}
	We call $\ASupp M$ the {\em atom support} of $M$.
\end{Def}

The following proposition is an expected property of ``supports''. 

\begin{Prop}\label{Prop:suppexact}
	Let $0\to L\to M\to N\to 0$ be an exact sequence. Then we have
	\begin{eqnarray*}
		\ASupp M=\ASupp L\cup\ASupp N.
	\end{eqnarray*}
\end{Prop}

\begin{proof}
	It is obvious that $\ASupp L\cup\ASupp N\subset\ASupp M$. Let $\overline{H}\in\ASupp M$. Then there exists $H'\in\overline{H}$ which is a quotient object of a subobject $M'$ of $M$. In a similar way to the proof of Proposition \ref{Prop:subcats} (4), we obtain the following commutative diagram:
	\begin{eqnarray*}
		\newdir^{ (}{!/-5pt/@^{(}}
		\xymatrix{
			0\ar[r] & B\ar[r] & H'\ar[r] & C\ar[r] & 0 \\
			0\ar[r] & L'\ar@{>>}[u]\ar@{^{ (}->}[d]\ar[r] & M'\ar@{>>}[u]\ar@{^{ (}->}[d]\ar[r] & N'\ar@{>>}[u]\ar@{^{ (}->}[d]\ar[r] & 0 \\
			0\ar[r] & L \ar[r]& M\ar[r] & N\ar[r] & 0.
		}
	\end{eqnarray*}
	
	In the case where $B\neq 0$, by Proposition \ref{Prop:monoformsub}, $B$ is a monoform subobject of $L'$ which is atom-equivalent to $H'$. Then $\overline{H}\in\ASupp L$.
	
	In the case where $B=0$, $H'$ is isomorphic to $C$. Then $\overline{H}\in\AAss N$.
\end{proof}

The associated atoms of a module is investigated in \cite{Storrer}. The following definition is in fact a generalization of that for modules.

\begin{Def}
	For an object $M$, we define a subclass $\AAss M$ of $\ASpec\mcA$ by
	\begin{eqnarray*}
		\AAss M=\{\overline{H}\in\ASpec\mcA\ |\text{ there exists }H'\in\overline{H}\\\text{which is a subobject of }M\}.
	\end{eqnarray*}
	We call an element in $\AAss M$ an {\em associated atom} of $M$.
\end{Def}

We obtain the following proposition, which is known as a property of associated primes.

\begin{Prop}\label{Prop:assexact}
	Let $0\to L\to M\to N\to 0$ be an exact sequence. Then we have
	\begin{eqnarray*}
		\AAss L\subset\AAss M\subset\AAss L\cup\AAss N.
	\end{eqnarray*}
\end{Prop}

\begin{proof}
	This is shown in a similar way to the proof of Proposition \ref{Prop:suppexact}.
\end{proof}

\begin{rem}
	For a monoform object $H$, it is clear that $\AAss H=\{\overline{H}\}$. Then by Theorem \ref{Thm:monoformfiltration} and Proposition \ref{Prop:assexact}, the number of the associated atoms of a nonzero noetherian object is nonzero and finite.
\end{rem}

Now we introduce a topology on $\ASpec \mcA$. 

\begin{Def}\label{Def:topology}
	We say that a subclass $\Phi$ of $\ASpec\mcA$ is {\em open} if for any $\overline{H}\in\Phi$, there exists $H'\in\overline{H}$ such that $\ASupp H'\subset\Phi$.
\end{Def}

\begin{Prop}\label{Prop:topology}
	The family of all the open subclasses of $\ASpec\mcA$ satisfies the axioms of topology.
\end{Prop}

\begin{proof}
	It is clear that $\emptyset,\ASpec\mcA$ are open and that the union of any family of open subclasses of $\mcA$ is also open.
	
	Let $\Phi_{1}$ and $\Phi_{2}$ be open subclasses of $\mcA$. For any $\overline{H}\in\Phi_{1}\cap\Phi_{2}$, there exist $H'_{1},H'_{2}\in\overline{H}$ such that $\ASupp H'_{1}\subset\Phi_{1}$, and $\ASupp H'_{2}\subset\Phi_{2}$. Since $H'_{1}$ is atom-equivalent to $H'_{2}$, there exists a common nonzero subobject $H'$ of $H'_{1}$ and $H'_{2}$. Then $H'\in\overline{H}$, and $\ASupp H'\subset \ASupp H'_{1}\cap \ASupp H'_{2}\subset\Phi_{1}\cap\Phi_{2}$. This means that $\Phi_{1}\cap\Phi_{2}$ is also open.
\end{proof}

We can state the topology on $\ASpec\mcA$ in the following way.

\begin{Prop}\label{Prop:topology2}
	A subclass $\Phi$ of $\ASpec\mcA$ is open if and only if for any $\overline{H}\in\Phi$, there exists an object $M$ such that $\overline{H}\in\ASupp M\subset\Phi$.
\end{Prop}

\begin{proof}
	If $\overline{H}\in\ASupp M$ for a monoform object $H$ and an object $M$, there exists $H'\in\overline{H}$ which is a subquotient of $M$. Then $\ASupp H'\subset\ASupp M$. This shows the claim.
\end{proof}

\section{Proof of Main Theorem 1}
\label{sec:main}

Throughout this section, let $\mcA$ be an abelian category. In this section, we prove Main Theorem 1. In order to do that, we define maps which give the one-to-one correspondence.

For a Serre subcategory $\mcX$ of $\mcA$, define a subclass $\ASupp\mcX$ of $\ASpec\mcA$ by
\begin{eqnarray*}
	\ASupp\mcX=\bigcup_{M\in\mcX}\ASupp M.
\end{eqnarray*}

For an open subclass $\Phi$ of $\ASpec\mcA$, define a subcategory $\ASupp^{-1}\Phi$ of $\mcA$ by
\begin{eqnarray*}
	\ASupp^{-1}\Phi=\{M\in\mcA\ |\ \ASupp M\subset\Phi\}.
\end{eqnarray*}

\begin{Lem}\label{Lem:Suppmaps}
	\begin{enumerate}
		\item For any Serre subcategory $\mcX$ of $\mcA$, $\ASupp\mcX$ is an open subclass of $\ASpec\mcA$.
		\item For any open subclass $\Phi$ of $\ASpec\mcA$, $\ASupp^{-1}\Phi$ is a Serre subcategory of $\mcA$.
	\end{enumerate}
\end{Lem}

\begin{proof}
	(1) This follows from Proposition \ref{Prop:topology2}.
	
	(2) This is immediate from Proposition \ref{Prop:suppexact}.
\end{proof}

The following lemma is a key to prove Main Theorem 1.

\begin{Lem}\label{Lem:suppSerre}
	Let $\mcX$ be a Serre subcategory of $\mcA$ and $M$ be a noetherian object such that $\ASupp M\subset\ASupp\mcX$. Then $M$ belongs to $\mcX$.
\end{Lem}

\begin{proof}
	Assume that $M$ does not belong to $\mcX$. Let $N$ be a maximal subobject of $M$ such that $M/N$ does not belong to $\mcX$. Then any proper quotient object of $M/N$ belongs to $\mcX$, and hence, by Theorem \ref{Thm:monoformness}, $M/N$ is a monoform object. Since $\overline{M/N}\in\ASupp M\subset\ASupp\mcX$, there exist an object $L$ in $\mcX$ and a subquotient $H$ of $L$ such that $M/N$ and $H$ have a common nonzero subobject $H'$. The quotient object of $M/N$ by $H'$ is of the form $M/N'$, where $N'$ is a subobject of $M$ such that $N\subsetneq N'$. By the maximality of $N$, $M/N'$ belongs to $\mcX$. Since we have the exact sequence
	\begin{eqnarray*}
		0\to H'\to \frac{M}{N}\to \frac{M}{N'}\to 0,
	\end{eqnarray*}
	and $M/N'$ and $H'$ belongs to $\mcX$, $M/N$ also belongs to $\mcX$. This is a contradiction. Therefore $M$ belongs to $\mcX$.
\end{proof}

Now we are ready to show Main Theorem 1. We say that an abelian category $\mcA$ is {\em noetherian} if any object in $\mcA$ is noetherian, and $\mcA$ is skeletally small. Note that $\ASpec\mcA$ forms a set if $\mcA$ is noetherian.\footnote{By replacing ``open subsets'' by ``open subclasses'', we still have Theorem \ref{Thm:main} for any abelian category consisting of noetherian objects which is not necessarily skeletally small.}

\begin{Thm}\label{Thm:main}
	Let $\mcA$ be a noetherian abelian category. Then the map $\mcX\mapsto\ASupp\mcX$ gives a one-to-one correspondence between Serre subcategories of $\mcA$ and open subsets of $\ASpec\mcA$. The inverse is given by $\Phi\mapsto\ASupp^{-1}\Phi$.
\end{Thm}

\begin{proof}
	Let $\mcX$ be a Serre subcategory of $\mcA$. It is obvious that $\mcX\subset\ASupp^{-1}(\ASupp\mcX)$. Assume that $M$ belongs to $\ASupp^{-1}(\ASupp\mcX)$. Since $\ASupp M\subset \ASupp\mcX$, by Lemma \ref{Lem:suppSerre}, $M$ belongs to $\mcX$. Therefore $\ASupp^{-1}(\ASupp\mcX)=\mcX$.
	
	Let $\Phi$ be an open subset of $\ASpec\mcA$. It is obvious that $\ASupp(\ASupp^{-1}\Phi)\subset\Phi$. For any $\overline{H}\in\Phi$, by the definition of open subsets of $\ASpec\mcA$, there exists $H'\in\overline{H}$ such that $\ASupp H'\subset\Phi$. Since $H'$ belongs to $\ASupp^{-1}\Phi$, we have $\overline{H}=\overline{H'}\in\ASupp H'\subset\ASupp(\ASupp^{-1}\Phi)$. Therefore $\ASupp(\ASupp^{-1}\Phi)=\Phi$.
\end{proof}

\section{In the case of locally noetherian Grothendieck categories}
\label{sec:locnoetherian}

Throughout this section, let $\mcA$ be a locally noetherian Grothendieck category, whose definition is as follows.

\begin{Def}
	\begin{enumerate}
		\item An abelian category $\mcA$ is called a {\em Grothendieck category} if $\mcA$ has a generator and arbitrary direct sums, and $\mcA$ satisfies the following condition: for any object $M$ in $\mcA$, any family $\{L_{\lambda}\}_{\lambda\in\Lambda}$ of subobjects of $M$ such that any finite subfamily of $\{L_{\lambda}\}_{\lambda\in\Lambda}$ has an upper bound $L_{\mu}$, and any subobject $N$ of $M$, we have
		\begin{eqnarray*}
			\left(\sum_{\lambda\in\Lambda}L_{\lambda}\right)\cap N=\sum_{\lambda\in\Lambda}(L_{\lambda}\cap N).
		\end{eqnarray*}
		\item A Grothendieck category $\mcA$ is called {\em locally noetherian} if there exists a set of generators of $\mcA$ consisting of noetherian objects.
	\end{enumerate}
\end{Def}

Denote by $\noeth\mcA$ the subcategory of noetherian objects in $\mcA$, which is in fact a Serre subcategory of $\mcA$.

The following well-known set-theoretic observations can be shown easily by standard arguments.

\begin{Prop}\label{Prop:subsetssmall}
	\begin{enumerate}
		\item Let $M$ be an object in $\mcA$.
		\begin{enumerate}
			\item The collection of all the subobjects of $M$ forms a set.
			\item The collection of all the quotient objects of $M$ forms a set.
		\end{enumerate}
		\item The category $\noeth\mcA$ is skeletally small.
	\end{enumerate}
\end{Prop}

By Proposition \ref{Prop:subsetssmall} (2), we deduce that $\noeth\mcA$ is a noetherian abelian category and that $\ASpec(\noeth\mcA)$ forms a set. The following proposition ensures that $\ASpec\mcA$ is also a set.

\begin{Prop}\label{Prop:noethatom}
	$\ASpec\mcA$ coincides with $\ASpec(\noeth\mcA)$ as a topological space.
\end{Prop}

\begin{proof}
	Any nonzero object in $\mcA$ has a nonzero noetherian subobject, and hence by Proposition \ref{Prop:monoformsub}, any element in $\ASpec\mcA$ is represented by an object in $\noeth\mcA$. Therefore the claim follows from the definition of the atom spectrum.
\end{proof}

In order to relate the subcategories of $\mcA$ and those of $\noeth\mcA$, we recall the following result.

\begin{Prop}\label{Prop:Serreloc}
	The map $\mcX\mapsto\mcX\cap\noeth\mcA$ gives a one-to-one correspondence between localizing subcategories of $\mcA$ and Serre subcategories of $\noeth\mcA$. The inverse map is given by $\mcY\mapsto\loc{\mcY}$.
\end{Prop}

\begin{proof}
	This follows from more general results \cite[Theorem 2.8]{Herzog} and \cite[Corollary 2.10]{Krause}.
\end{proof}

Now we have the classification of localizing subcategories of $\mcA$.

\begin{Thm}\label{Thm:locSerreatom}
	Let $\mcA$ be a locally noetherian Grothendieck category. Then there exist one-to-one correspondences between localizing subcategories of $\mcA$, Serre subcategories of $\noeth\mcA$, and open subsets of $\ASpec\mcA$.
\end{Thm}

\begin{proof}
	Apply Theorem \ref{Thm:main} to $\noeth\mcA$, and use Proposition \ref{Prop:noethatom} and Proposition \ref{Prop:Serreloc}.
\end{proof}

The following proposition about atom supports and associated atoms is well-known as properties of supports and associated primes in the commutative ring theory.

\begin{Prop}\label{Prop:suppassdirectsum}
	For any family $\{M_{\lambda}\}_{\lambda\in\Lambda}$ of objects in $\mcA$, we have
	\begin{eqnarray*}
		\ASupp \bigoplus_{\lambda\in\Lambda}M_{\lambda}=\bigcup_{\lambda\in\Lambda}\ASupp M_{\lambda}
	\end{eqnarray*}
	and
	\begin{eqnarray*}
		\AAss \bigoplus_{\lambda\in\Lambda}M_{\lambda}=\bigcup_{\lambda\in\Lambda}\AAss M_{\lambda}.
	\end{eqnarray*}
\end{Prop}

\begin{proof}
	It is clear that $\bigcup_{\lambda\in\Lambda}\ASupp M_{\lambda}\subset\ASupp \bigoplus_{\lambda\in\Lambda}M_{\lambda}$. Let $\overline{H}\in\ASupp\bigoplus_{\lambda\in\Lambda}M_{\lambda}$. Then there exists $H'\in\overline{H}$ which is a subquotient of $\bigoplus_{\lambda\in\Lambda}M_{\lambda}$. Take a nonzero noetherian subobject $H''$ of $H'$. Then by Proposition \ref{Prop:monoformsub}, $H''$ is also a monoform object. By Proposition \ref{Prop:subcats} (1), $H''$ is a quotient object of a subobject $L$ of $\bigoplus_{\lambda\in\Lambda}M_{\lambda}$. There exists a noetherian subobject $L'$ of $L$ such that the composite $L'\hookrightarrow L\twoheadrightarrow H''$ is still an epimorphism. Since $L'$ is a noetherian subobject of $\bigoplus_{\lambda\in\Lambda}M_{\lambda}$, there exist $\lambda_{1},\ldots,\lambda_{n}\in\Lambda$ such that $L'\subset\bigoplus_{i=1}^{n}M_{\lambda_{i}}$, and hence, by Proposition \ref{Prop:suppexact},
	\begin{eqnarray*}
		\overline{H}=\overline{H'}=\overline{H''}\in\ASupp\bigoplus_{i=1}^{n}M_{\lambda_{i}}=\bigcup_{i=1}^{n}\ASupp M_{\lambda_{i}}\subset\bigcup_{\lambda\in\Lambda}\ASupp M_{\lambda}.
	\end{eqnarray*}
	Therefore $\ASupp \bigoplus_{\lambda\in\Lambda}M_{\lambda}\subset\bigcup_{\lambda\in\Lambda}\ASupp M_{\lambda}$.
	
	Similarly, we can show that $\AAss \bigoplus_{\lambda\in\Lambda}M_{\lambda}=\bigcup_{\lambda\in\Lambda}\AAss M_{\lambda}$ by Proposition \ref{Prop:assexact}.
\end{proof}

The atom spectrum of $\mcA$ is in fact homeomorphic to the Ziegler spectrum of $\mcA$. The definition of the Ziegler spectrum was originally given by Ziegler \cite{Ziegler}. We use the following definition as in \cite{Herzog}.

\begin{Def}\label{Def:Zieglerspectrum}
	Let $\mcA$ be a locally noetherian Grothendieck category. Denote the collection of all the isomorphism classes of indecomposable injective objects in $\mcA$ by $\Zg\mcA$. For a noetherian object $M$, set
	\begin{eqnarray*}
		\mcO(M)=\{I\in\Zg\mcA\ |\ \Hom_{\mcA}(M,I)\neq 0\}.
	\end{eqnarray*}
	We define a topology on $\Zg\mcA$ by taking
	\begin{eqnarray*}
		\{\mcO(M)\ |\ M\text{ is a noetherian object in }\mcA\}
	\end{eqnarray*}
	as a basis of open subclasses. The topological space $\Zg\mcA$ is called the {\em Ziegler spectrum} of $\mcA$.
\end{Def}

Recall that any object $M$ in a Grothendieck category $\mcA$ has its injective hull $E(M)$ in $\mcA$.

\begin{Lem}\label{Lem:atomenv}
	Monoform objects $H$ and $H'$ in $\mcA$ are atom-equivalent if and only if $E(H)\cong E(H')$.
\end{Lem}

\begin{proof}
	Assume that $H$ and $H'$ have a common nonzero subobject $H''$. Since $H$ and $H'$ are uniform by Proposition \ref{Prop:monoformuniform}, we have $E(H)\cong E(H'')\cong E(H')$.
	
	Conversely, assume that $E(H)\cong E(H')$. Then $H$ and $H'$ are nonzero subobjects of $E(H)$. Since $H$ is uniform, $E(H)$ is also uniform. Hence $H$ and $H'$ have a common nonzero subobject.
\end{proof}

According to Lemma \ref{Lem:atomenv}, we can define the {\em injective hull} $\overline{E}(\overline{H})$ of an atom $\overline{H}$ in $\mcA$ by $\overline{E}(\overline{H})=E(H)$. This operation gives a homeomorphism between the atom spectrum of $\mcA$ and the Ziegler spectrum of $\mcA$.

\begin{Thm}\label{Thm:atomspecZg}
	Let $\mcA$ be a locally noetherian Grothendieck category.
	\begin{enumerate}
		\item The map $\overline{E}(\cdot)$ gives a one-to-one correspondence between atoms in $\mcA$ and isomorphism classes of indecomposable injective objects in $\mcA$.
		
		\item The map $\overline{E}(\cdot)$ sends the atom support $\ASupp M$ of a noetherian object $M$ in $\mcA$ to $\mcO(M)$.
	
		\item $\ASpec\mcA$ is homeomorphic to $\Zg\mcA$.
	\end{enumerate}
\end{Thm}

\begin{proof}
	(1) Lemma \ref{Lem:atomenv} shows that $\overline{E}(\cdot)$ is injective. Let $I$ be an indecomposable injective object in $\mcA$. By applying Theorem \ref{Thm:monoformfiltration} to a nonzero noetherian subobject of $I$, we obtain a monoform subobject $H$ of $I$. Then $\overline{E}(\overline{H})=I$, and hence $\overline{E}(\cdot)$ is surjective.
	
	(2) Let $H$ be a monoform object and $I=\overline{E}(\overline{H})$. Assume that $\overline{H}\in\ASupp M$. Then there exists $H'\in\overline{H}$ which is a subobject of a quotient object $N$ of $M$. Since $I$ is injective, there exists a nonzero morphism from $N$ to $I$ such that the diagram
	\begin{eqnarray*}
		\newdir^{ (}{!/-5pt/@^{(}}
		\xymatrix{
			N\ar[dr] & \\
			H\ar@{^{ (}->}[u]\ar@{^{ (}->}[r] & I
		}
	\end{eqnarray*}
	commutes. The nonzero composite $M\twoheadrightarrow N\to I$ implies that $I\in\mcO(M)$.
	
	Conversely, assume that $I\in\mcO(M)$. Then there exists a nonzero morphism from $M$ to $I$. Denote its image by $B$. Since $B$ is nonzero, and $I$ is uniform, by Proposition \ref{Prop:monoformsub}, $B\cap H$ is a monoform subobject of $I$ which is atom-equivalent to $H$. Since $B\cap H$ is a subquotient of $M$, we have $\overline{H}\in\ASupp M$.
	
	(3) By Proposition \ref{Prop:topology2}, $\overline{E}(\cdot)$ is a homeomorphism.
\end{proof}

In the rest of this section, we note that the assumption of noetherian in Proposition \ref{Prop:quotuniform} and Proposition \ref{Prop:monoformsum} can be dropped in the case where $\mcA$ is a Grothendieck category.

\begin{Prop}\label{Prop:generalizedquotuniform}
	Let $M$ be an object in $\mcA$ and $L$ be a uniform subobject of $M$. Then there exists a subobject $N$ of $M$ such that the composite $L\hookrightarrow M\twoheadrightarrow M/N$ is a monomorphism, and $M/N$ is a uniform object.
\end{Prop}

\begin{proof}
	Define a collection $\mcS$ of subsets of $M$ by
	\begin{eqnarray*}
		\mcS=\{M'\in\mcA\ |\ M'\text{ is a subobject of }M\text{ such that }L\cap M'=0\}.
	\end{eqnarray*}
	By Proposition \ref{Prop:subsetssmall} (1) (a), $\mcS$ is a set. Let $\{N_{\lambda}\}_{\lambda\in\Lambda}$ be a totally ordered subset of $\mcS$. Then by the axiom of the Grothendieck category,
	\begin{eqnarray*}
		L\cap\sum_{\lambda\in\Lambda}N_{\lambda}=\sum_{\lambda\in\Lambda}(L\cap N_{\lambda})=0,
	\end{eqnarray*}
	and hence $\sum_{\lambda\in\Lambda}N_{\lambda}$ is an upper bound of $\{N_{\lambda}\}_{\lambda\in\Lambda}$ in $\mcS$. Therefore by Zorn's lemma, there exists a maximal element $N$ in $\mcS$, and the assertion is shown in the same way as the proof of Proposition \ref{Prop:quotuniform}.
\end{proof}

\begin{Prop}\label{Prop:generalizedmonoformsum}
	Let $M$ be a uniform object in $\mcA$ and $H_{1},H_{2}$ be monoform subobjects of $M$. Then $H_{1}+H_{2}$ is a monoform subobject of $M$.
\end{Prop}

\begin{proof}
	This follows from the proof of Proposition \ref{Prop:monoformsum}. Use Proposition \ref{Prop:generalizedquotuniform} instead of Proposition \ref{Prop:quotuniform}.
\end{proof}

\section{In the case of right noetherian rings}
\label{sec:ring}

In this section, for a right noetherian ring $R$, we describe the atom spectrum of $\Mod R$ in terms of right ideals of $R$.

\begin{Def}\label{Def:comonoformrightideal}
	Call a right ideal $\mfp$ of $R$ a {\em comonoform right ideal} of $R$ if $R/\mfp$ is a monoform module.
\end{Def}

Note that any maximal right ideal of $R$ is a comonoform right ideal of $R$.

Denote the atom spectrum of $\Mod R$ by $\ASpec R$, and call it the {\em atom spectrum} of $R$. For a comonoform right ideal $\mfp$ of $R$, denote the atom $\overline{R/\mfp}$ by $\widetilde{\mfp}$.

\begin{Prop}
	$\ASpec R=\{\widetilde{\mfp}\ |\ \mfp\text{ is a comonoform right ideal of }R\}$.
\end{Prop}

\begin{proof}
	Any monoform module $H$ has a nonzero submodule generated by one element $x\in H$. Then $xR$ is isomorphic to $R/\Ann_{R}(x)$. Since $xR$ is also monoform by Proposition \ref{Prop:monoformsub}, $\Ann_{R}(x)$ is a comonoform right ideal of $R$.
\end{proof}

Then we obtain the following corollary.

\begin{Cor}\label{Cor:ringcorresp}
	Let $R$ be a right noetherian ring. Then there exist one-to-one correspondences between localizing subcategories of $\Mod R$, Serre subcategories of $\mod R$, and open subsets of $\ASpec R$.
\end{Cor}

\begin{proof}
	Apply Theorem \ref{Thm:locSerreatom} to $\Mod R$.
\end{proof}

Now we compare comonoform right ideals to completely prime right ideals in \cite{Reyes}.

\begin{Def}
	A right ideal $I\subsetneq R$ is called a {\em completely prime right ideal} of $R$ if for any $a,b\in R$, $aI\subset I$ and $ab\in I$ imply $a\in I$ or $b\in I$.
\end{Def}

The following proposition is shown in \cite{Reyes}.

\begin{Prop}\label{Prop:comonoformcp}
	Any comonoform right ideal of $R$ is a completely prime right ideal of $R$.
\end{Prop}

\begin{proof}
	Let $\mfp$ be a comonoform right ideal of $R$. Assume that $\mfp$ is not a completely prime right ideal of $R$. Then there exist $a,b\in R$ such that $a\mfp\subset\mfp$, $ab\in \mfp$, $a\not\in\mfp$, and $b\not\in\mfp$. Define a morphism $f:R/\mfp\to R/\mfp$ by $f(c+\mfp)=ac+\mfp$. $a\mfp\subset\mfp$ implies that $f$ is well-defined. Since $a\not\in\mfp$, the image of $f$ is a nonzero submodule of $R/\mfp$. Since $b\not\in\mfp$, and $ab\in\mfp$, the kernel of $f$ is a nonzero submodule of $R/\mfp$, and hence it is of the form $I/\mfp$, where $I$ is a right ideal of $R$ which satisfies $\mfp\subsetneq I$. Then we have
	\begin{eqnarray*}
		\frac{R}{\mfp}\hookleftarrow \Im f\cong \frac{R/\mfp}{\Ker f}\cong \frac{R}{I},
	\end{eqnarray*}
	and this contradicts the fact that $R/\mfp$ is a monoform module. Therefore $\mfp$ is a completely prime right ideal of $R$.
\end{proof}

The converse of Proposition \ref{Prop:comonoformcp} fails. In fact, it is observed in \cite{Reyes} that even if $I$ is a completely prime right ideal of $R$, $R/I$ is not necessarily a uniform module. For any comonoform right ideal $\mfp$, $R/\mfp$ is, however, a uniform module. This is an advantage of considering comonoform right ideals.

On the other hand, the following fact shows that there exist ``sufficiently many'' comonoform right ideals.

\begin{Cor}\label{Cor:comonoformfiltration}
	Let $R$ be a right noetherian ring. Then for any finitely generated module $M$, there exist a filtration
	\begin{eqnarray*}
		0=L_{0}\subset L_{1}\subset\cdots\subset L_{n}=M
	\end{eqnarray*}
	and comonoform right ideals $\mfp_{1},\ldots,\mfp_{n}$ of $R$ such that $L_{i}/L_{i-1}\cong R/\mfp_{i}$ for any $i=1,\ldots,n$.
\end{Cor}

\begin{proof}
	Assume that for some $k\in\mbZ_{\geq 0}$, there exist a sequence
	\begin{eqnarray*}
		0=L_{0}\subset L_{1}\subset\cdots\subset L_{k}\subsetneq M
	\end{eqnarray*}
	of submodules of $M$ and comonoform right ideals $\mfp_{1},\ldots,\mfp_{k}$ of $R$ such that $L_{i}/L_{i-1}\cong R/\mfp_{i}$ for any $i=1,\ldots,k$. Then by Theorem \ref{Thm:monoformfiltration}, $M/L_{k}$ has a monoform submodule $L_{k+1}/L_{k}$. By Proposition \ref{Prop:monoformsub}, we can assume that $L_{k+1}/L_{k}$ is generated by one element, and hence $L_{k+1}/L_{k}\cong R/\mfp_{k+1}$ for some comonoform right ideal $\mfp_{k+1}$ of $R$. By repeating these operations, we obtain $n\in\mbZ_{\geq 0}$ such that $L_{n}=M$ since $M$ is noetherian.
\end{proof}

\section{In the case of commutative noetherian rings}
\label{sec:commutative}

In this section, let $R$ be a commutative noetherian ring. We show that the atom spectrum coincides with the prime spectrum as a set. A {\em comonoform ideal} means a comonoform right ideal of a commutative ring. 

\begin{Prop}
	An ideal of $R$ is a comonoform ideal of $R$ if and only if it is a prime ideal of $R$.
\end{Prop}

\begin{proof}
	Any comonoform ideal of $R$ is a prime ideal of $R$ by Proposition \ref{Prop:comonoformcp}. Let $\mfp$ be a prime ideal of $R$. If $\mfp$ is not a comonoform ideal of $R$, there exist ideals $I$, $J$, and $J'$ of $R$ such that $\mfp\subsetneq I$, $\mfp\subsetneq J\subsetneq J'$, and $I/\mfp\cong J'/J$. Let $a+\mfp$ be a nonzero element of $I/\mfp$, and take the corresponding element $b+J\in J'/J$. Since $\mfp$ is a prime ideal of $R$, we have $\Ann_{R}(a+\mfp)=\mfp$. On the other hand, we have $J\subset\Ann_{R}(b+J)$. This contradicts $\mfp\subsetneq J$. Therefore $\mfp$ is a comonoform ideal of $R$.
\end{proof}

We have an equivalence relation between comonoform right ideals, which is induced by the atom equivalence. If $R$ is commutative, it is in fact a trivial relation. We say that a subset $\Phi$ of $\Spec R$ is {\em closed under specialization} if for any $\mfp,\mfq\in\Spec R$, $\mfp\subset\mfq$ and $\mfp\in\Phi$ imply $\mfq\in\Phi$.

\begin{Prop}\label{Prop:commatom}
	\begin{enumerate}
		\item For any comonoform ideals $\mfp$ and $\mfq$, it holds that $\widetilde{\mfp}=\widetilde{\mfq}$ if and only if $\mfp=\mfq$. Hence $\ASpec R=\Spec R$ as a set.
		
		\item A subset $\Phi$ of $\ASpec R$ is open if and only if $\Phi$ is specialization-closed subset of $\Spec R$.
	\end{enumerate}
\end{Prop}

\begin{proof}
	(1) For any nonzero element $a+\mfp$ of $R/\mfp$, we have $\Ann_{R}(a+\mfp)=\mfp$. Hence there exists no common nonzero submodule of $R/\mfp$ and $R/\mfq$ if $\mfp\neq\mfq$.
	
	(2) According to (1), we identify the atom $\widetilde{\mfp}$ with $\mfp$. By Proposition \ref{Prop:monoformsub} and (1), for each comonoform ideal $\mfp$ of $R$, any nonzero submodule of $R/\mfp$ generated by one element is isomorphic to $R/\mfp$. Therefore
	\begin{eqnarray*}
		\ASupp\frac{R}{\mfp}=\{\mfq\in\Spec R\ |\ \mfp\subset\mfq\}=\Supp\frac{R}{\mfp}.
	\end{eqnarray*}
\end{proof}

In particular, we obtain the following Gabriel's theorem.

\begin{Cor}[Gabriel \cite{Gabriel}]\label{Cor:commringcorresp}
	Let $R$ be a commutative noetherian ring. Then there exist one-to-one correspondences between localizing subcategories of $\Mod R$, Serre subcategories of $\mod R$, and specialization-closed subsets of $\Spec R$.
\end{Cor}

\begin{proof}
	This is immediate from Corollary \ref{Cor:ringcorresp} and Proposition \ref{Prop:commatom}.
\end{proof}

\begin{rem}
	It follows from Theorem \ref{Thm:atomspecZg} that the topology on $\Spec R$ of the atom spectrum coincides with the Ziegler spectrum. Furthermore, as in \cite{GarkushaPrest}, it also coincides with the Hochster dual of the Zariski topology on $\Spec R$.
\end{rem}

For any module $M$, by the definition, $\AAss M=\Ass M$. We show that $\ASupp M=\Supp M$. Note that $\ASupp (R/\mfp)=\Supp (R/\mfp)$ for any $\mfp\in\Spec R$. By Corollary \ref{Cor:comonoformfiltration} and Proposition \ref{Prop:suppexact}, we have $\ASupp L=\Supp L$ for any finitely generated module $L$. For an arbitrary module $M$, there exists a family $\{L_{\lambda}\}_{\lambda\in\Lambda}$ of finitely generated submodule of $M$ such that $M=\sum_{\lambda\in\Lambda}L_{\lambda}$. By Proposition \ref{Prop:suppexact} and Proposition \ref{Prop:suppassdirectsum},
\begin{eqnarray*}
	\bigcup_{\lambda\in\Lambda}\ASupp L_{\lambda}\subset\ASupp\sum_{\lambda\in\Lambda}L_{\lambda}\subset\ASupp\bigoplus_{\lambda\in\Lambda}L_{\lambda}=\bigcup_{\lambda\in\Lambda}\ASupp L_{\lambda}.
\end{eqnarray*}
Therefore
\begin{eqnarray*}
	\ASupp M=\bigcup_{\lambda\in\Lambda}\ASupp L_{\lambda}=\bigcup_{\lambda\in\Lambda}\Supp L_{\lambda}=\Supp M.
\end{eqnarray*}

\section{In the case of right artinian rings}
\label{sec:artinian}

In this section, let $R$ be a right artinian ring. In this case, we can obtain an explicit description of $\ASpec R$. In the following proposition, we give a characterization of monoform modules.

\begin{Prop}\label{Prop:artinianmonoform}
	A finitely generated module $M$ is a monoform module if and only if it has a simple socle $S$, and there exists no other composition factor of $M$ which is isomorphic to $S$.
\end{Prop}

\begin{proof}
	Assume that $M$ is monoform. Since $R$ is right artinian, the socle $S$ of $M$ is nonzero. By Proposition \ref{Prop:monoformuniform}, $M$ is a uniform module, and hence $S$ is simple. If $M$ has another composition factor which is isomorphic to $S$, so does $M/S$. Then there exists a submodule $N$ of $M$ such that $S\subset N$, and $M/N$ has a submodule which is isomorphic to $S$. This contradicts the fact that $M$ is a monoform module. Therefore $M$ does not have another composition factor which is isomorphic to $S$.
	
	Conversely, assume that $M$ has a simple socle $S$ and that there exists no other composition factor of $M$ which is isomorphic to $S$. Any nonzero submodule $N$ of $M$ contains $S$, and hence $M/N$ does not have a submodule isomorphic to $S$. This means that $M$ is a monoform module.
\end{proof}

Note that monoform modules $H$ and $H'$ are atom-equivalent if and only if their simple socles are isomorphic to each other.

\begin{Prop}\label{Prop:artinian}
	Let $\mfm_{1},\ldots,\mfm_{n}$ be maximal right ideals of $R$ such that
	\begin{eqnarray*}
		\left\{\frac{R}{\mfm_{1}},\ldots,\frac{R}{\mfm_{n}}\right\}
	\end{eqnarray*}
	is a maximal set of simple modules which are pairwise non-isomorphic. Then
	\begin{eqnarray*}
		\ASpec R=\{\widetilde{\mfm_{1}},\ldots,\widetilde{\mfm_{n}}\}
	\end{eqnarray*}
	with the discrete topology, and $\widetilde{\mfm_{i}}\neq\widetilde{\mfm_{j}}$ if $i\neq j$.
\end{Prop}

\begin{proof}
	Any maximal right ideal is a comonoform right ideal, and any monoform module is atom-equivalent to its simple socle. $\ASupp(R/\mfm_{i})=\{\widetilde{\mfm_{i}}\}$.
\end{proof}

\begin{ex}\label{ex:lowertri}
	Let $R$ be the ring of lower triangular matrices over a field $K$, that is,
	\begin{eqnarray*}
		R=\begin{bmatrix} K & 0 \\ K & K \end{bmatrix}.
	\end{eqnarray*}
	Then all the right ideals of $R$ are
	\begin{eqnarray*}
		0, \begin{bmatrix} 0 & 0 \\ K & 0 \end{bmatrix}, \mfp_{a}=K\begin{bmatrix} 1 & 0 \\ a & 0 \end{bmatrix} (a\in K), \mfm_{1}=\begin{bmatrix} 0 & 0 \\ K & K \end{bmatrix}, \mfm_{2}=\begin{bmatrix} K & 0 \\ K & 0 \end{bmatrix}, R.
	\end{eqnarray*}
	All the comonoform right ideals of $R$ are
	\begin{eqnarray*}
		\mfp_{a} (a\in K), \mfm_{1}, \mfm_{2}.
	\end{eqnarray*}
	Since
	\begin{eqnarray*}
		\frac{R}{\mfp_{a}}\cong\begin{bmatrix} K & K \end{bmatrix}, \frac{R}{\mfm_{1}}\cong\begin{bmatrix} K & 0 \end{bmatrix}, \frac{R}{\mfm_{2}}\cong\frac{\begin{bmatrix} K & K \end{bmatrix}}{\begin{bmatrix} K & 0 \end{bmatrix}},
	\end{eqnarray*}
	we have $\widetilde{\mfp_{a}}=\widetilde{\mfm_{1}}\neq\widetilde{\mfm_{2}}$. Therefore all the Serre subcategories of $\mod R$ are $\{\text{zero objects}\}$, $\Serre{R/\mfm_{1}}$, $\Serre{R/\mfm_{2}}$, and $\mod R$.
\end{ex}

\section{Acknowledgments}

The author would like to express his deep gratitude to his supervisor Osamu Iyama for his elaborated guidance. The author thanks Ryo Takahashi for his valuable comments.



\end{document}